%% file: iterated_altans_nb.tex
\documentclass[a4paper,12pt]{article}
\usepackage[utf8]{inputenc}
\usepackage{amsmath,amssymb,amsthm,amsfonts}
\usepackage{euscript,graphicx}
\usepackage{longtable}
\usepackage{tikz}
\usetikzlibrary{calc}
\usetikzlibrary{decorations.pathmorphing}
\usetikzlibrary{arrows}
\usepackage{url}
\usepackage{subfigure}
\usepackage{paralist}

\usepackage[text={16cm,24cm},centering]{geometry} 

\newtheorem{theorem}{Theorem}
\newtheorem{proposition}[theorem]{Proposition}
\newtheorem{lemma}[theorem]{Lemma}
\newtheorem{corollary}[theorem]{Corollary}

\newtheorem{definition}[theorem]{Definition}

\theoremstyle{definition}
\newtheorem{example}[theorem]{Example}

\title{Iterated altans and their properties}
\author{Nino Ba\v si\' c\footnote{\textit{Faculty of Mathematics and Physics, University of Ljubljana, Slovenia}, e-mail: \texttt{nino.basic@fmf.uni-lj.si}} \and Toma\v z Pisanski\footnote{\textit{FAMNIT, University of Primorska, Slovenia}, e-mail: \texttt{tomaz.pisanski@upr.si}}}
\date{April 30, 2015}

\begin{document}

\nocite{*} 

\maketitle

\begin{abstract}
Recently a class of molecular graphs, called \emph{altans}, became a focus of attention of several theoretical chemists and mathematicians. In this paper we study primary iterated altans and show, among other things, their connections with nanotubes and nanocaps. The question of classification of bipartite altans is also addressed. Using the results of Gutman we are able to enumerate Kekul\'{e} structures of several nanocaps of arbitrary length.
\end{abstract}


\vspace{0.5\baselineskip}
\noindent \textbf{Math.\ Subj.\ Class.\ (2010):} 92E10, 05C90

\section{Introduction}
Altans were first introduced as special planar systems, obtained from benzenoids by attachment of a ring to all outer vertices of valence two \cite{gutman2014,monacozanasi2009}, in particular in connection with concentric decoupled nature of the ring currents (see papers by Zanasi et al.\  \cite{monacomemoli2013,monacozanasi2009} and also Mallion and Dickens \cite{dickensmallion2014a,dickensmallion2014}). The graph-theoretical approach to ring current was initiated by Milan Randi\'{c} in 1976 \cite{randic1976}. It was also studied by Gomes and Mallion \cite{gomes1979}. Full description is provided for instance in  \cite{randic2003, randic2010}. Moreover, see paper \cite{fowlermyrvold2011} by Patrick Fowler and Wendy Myrvold on `The Anthracene Problem'. 

Later altans were generalized by Ivan Gutman \cite{gutman2014a} to arbitrary graphs. We essentially follow Gutman's approach. Our point of departure is a \emph{peripherally rooted graph}, i.e. an ordered pair $(G, S)$ which consists of an arbitrary graph $G$ and a cyclically ordered subset $S$ of its vertices, called the \emph{peripheral root}. 

Let $n$ denote the order of $G$ and let $k$ denote the cardinality of $S$. Assume that $V(G) = \{0, 1, \ldots, n-1\}$.
The operation $A(G, S)$ maps the pair $(G, S)$ to a new pair $(G_1, S_1)$ as follows:
Let $S_0 = \{n, n+1, \ldots, n+k-1\}$ and $S_1 = \{n+k, n+k+1, \ldots, n+2k-1\}$.
Let the vertex set of $G$ be augmented by $S_0 \cup S_1$. Through the vertices $S_0 \cup S_1$, we construct a \emph{peripheral} cycle graph $C$ of length $2k$ in the cyclic order
$$
(n, n+k, n+1, n+k+1, n+2, \ldots, n+k-1, n+2k-1, n).
$$
Finally, we attach $C$ to $G$ by $k$ edges between $S$ and $S_0$ of the form $(s_i, n+i)$, $0 \leq i < k$, where $s_i$ is the $i$-th vertex of $S$. The vertices of $C$ that have valence 2  in the final construction form are exactly the ones originating from $S_1$ and are the new peripheral root of the altan. The new peripheral root, $S_1$, is ordered in the natural way.

\begin{example}
\label{ex1}
A bipartite graph may give rise to non-bipartite or bipartite altans. Let $G = C_6$ and $S = (0, 1, 2, 3, 4, 5)$.
Graph $G$ and the non-bipartite $A(G,S)$ are depicted in Figure~\ref{fig:example1}.

\begin{figure}[!htbp]
\centering
\subfigure[$G = C_6$, $S = (0, 1, 2, 3, 4, 5)$]{
\label{fig:example1sub1}
\input{figure_1.tikz}
}
\quad\quad
\subfigure[$A(G, S) = (G_1, S_1)$. $S_1 = (12, 13, 14, 15, 16, 17)$ consists of vertices of valence 2 in the natural cyclic order.]{
\label{fig:example1sub2}
\input{figure_1a.tikz}
}
\caption{Altan of the benzene is not bipartite since $A(G,S)$ contains pentagons.}
\label{fig:example1}
\end{figure}
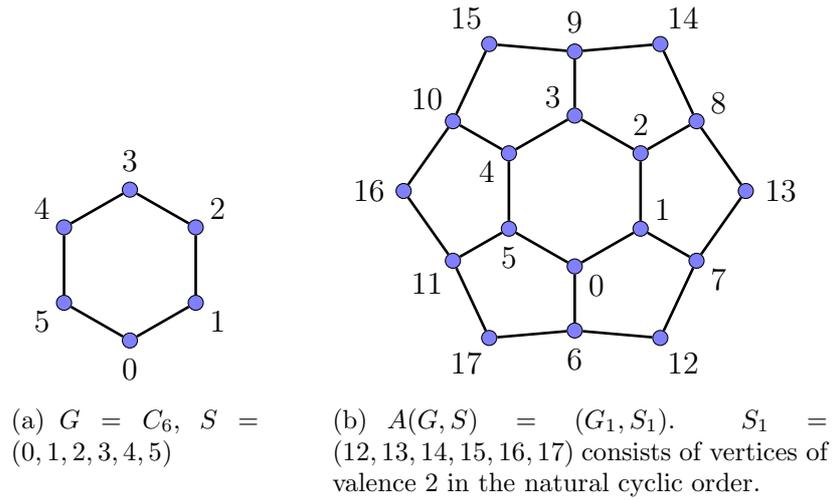
\end{example}

\begin{example}
\label{ex2}
\noindent The altan of the graph $G$ in Figure~\ref{fig:example2} is bipartite.

\begin{figure}[!htbp]
\centering
\subfigure[$G$, $S = (a, b, c)$]{
\label{fig:example2sub1}
\input{figure_2.tikz}
}
\subfigure[$G$, $S_1 = (a', b', c')$]{
\label{fig:example2sub2}
\input{figure_2a.tikz}
}
\caption{Since $G$ is bipartite and all vertices of $S$ 
belong to the same set of bipartition, $A(G,S)$ is bipartite as well.}
\label{fig:example2}
\end{figure}
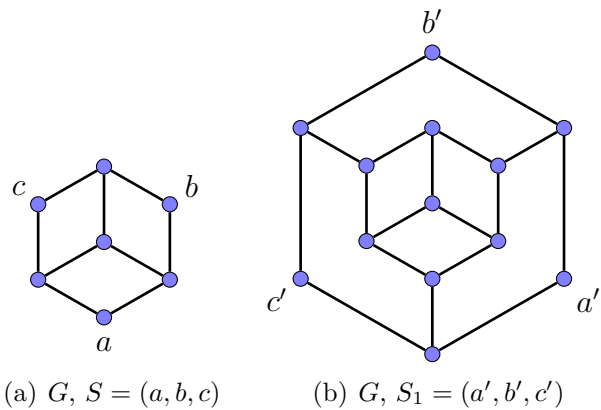
\end{example}
Note that the altan in Example~\ref {ex1} is non-bipartite while the one in Example~\ref{ex2} is bipartite. We can classify bipartite altans.

\begin{theorem}
\label{thm1}
Let $(G, S)$ be a graph $G$ with a peripheral root $S$. The altan $A(G, S)$ is bipartite if and only if
\begin{compactenum}
\renewcommand{\labelenumi}{\alph{enumi})}
\item $G$ is bipartite and
\item members of $S$ belong to the same bipartition set.
\end{compactenum}
\end{theorem}

\begin{proof}
First, we will show that conditions a) and b) imply that $A(G, S) = (G_1, S_1)$ is bipartite. Let $G$ be bipartite and let us colour its vertices black and white. We may assume that $S$ has all of its vertices coloured black. Their adjacent vertices, $S_0$, can be coloured white. This can be further extended, by colouring $S_1$ black, to a proper black and white vertex colouring of $G_1$. Hence $G_1$ is bipartite. Furthermore, $(G_1, S_1)$ also satisfies the conditions a) and b) of the Theorem.

Now we will show the vice versa. Let $A(G, S)$ be bipartite. 
Graph $G$ is a subgraph of $G_1$, so it is bipartite as well.
If $G$ is bipartite but not all vertices of $S$ are coloured with the same colour, then two consecutive vertices of $S$, say $u$ and $v$, would be coloured differently. Recall that vertices of $S$ are cyclically ordered. Hence there is a $u,v$-path in $G$ of odd length. By attaching the cycle $C$ to $G$ to form $A(G,S)$ it is possible to connect $u$ to $v$ by a path of length $4$.
This means there is a cycle of odd length in graph $G_1$, a contradiction. 
\end{proof}
From the definion of the altan operation it follows that we may repeat it several times. Let $A^n(G, S)$ denote the $n$-th altan of $(G, S)$, i.e.\ $\underbrace{A(A(\cdots A}_{n}(G, S)\cdots))$. We obtain the following consequence of Theorem~\ref{thm1}.

\begin{corollary}
\label{cor1}
Let $(G, S)$ be a graph with peripheral root $S$ and let $n \geq 1$ be an arbitrary integer. $A^n(G, S)$ is bipartite if and only if  $(G, S)$ satisfies the conditions of Theorem \ref{thm1}.
\end{corollary}

\begin{proof}
It follows by induction. The basis of induction is given by Theorem \ref{thm1}.
\end{proof}

\section{Altans of benzenoid systems}

Let $B$ be a finite benzenoid system. The altan of $B$ is assumed to have the cyclically ordered peripheral vertices, $S$, of valence 2. Their order is obtained by traversing $B$ along its perimeter.

\begin{theorem}
\label{thm2}
For any finite benzenoid $B$ its altan $A(B)$ is non-bipartite.
\end{theorem}

\begin{proof}
By \cite{gutman1989} each finite benzenoid has two consecutive peripheral vertices of valence 2, say $u$ and $v$. Both vertices $u$ and $v$ are attached to the outer cycle $C$ of $A(B)$ to vertices $u'$ and $v'$ respectively that are two-apart in $C$. Let $w'$ be the vertex on $C$ adjacent to both $u'$ and $v'$. Vertices $uvv'w'u'$ form a cycle of length 5. Altan $A(B)$ is not bipartite.

Note that the result follows also from Theorem 1. Although $B$ is bipartite, the corresponding peripheral root $S$ is not all coloured with the same color.
\end{proof}

For a bipartite graph $G$ and a peripheral root $S$ we may define partition $S_b$ and $S_w$ with black and white coloured sets. For a connected $G$ the partition is unique. We can consider two bipartite altans $A(G, S_b)$ and $A(G, S_w)$. In case of benzenoids the definition is natural and the bipartite altans are determined by the benzenoid itself.

\begin{example}
Black and white altans may be isomorphic (see Figures~\ref{fig:example3} and \ref{fig:example4}) or not (see Figure~\ref{fig:example4a}.)
\begin{figure}[!htbp]
\centering
\subfigure[$B$]{
\label{fig:example3sub1}
\input{figure_3.tikz}
}
\subfigure[$A_w(B) \cong A_b(B)$]{
\label{fig:example3sub2}
\input{figure_3a.tikz}
}
\caption{Black and white altans are isomorphic: $A_w(B) \cong A_b(B)$.}
\label{fig:example3}
\end{figure}
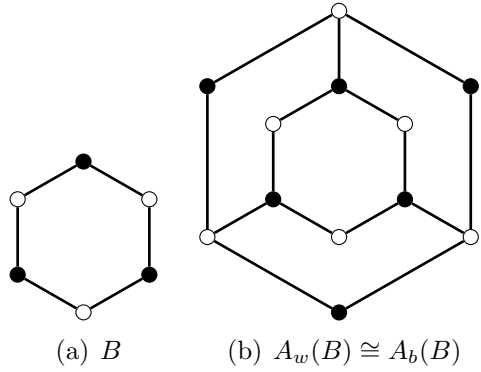



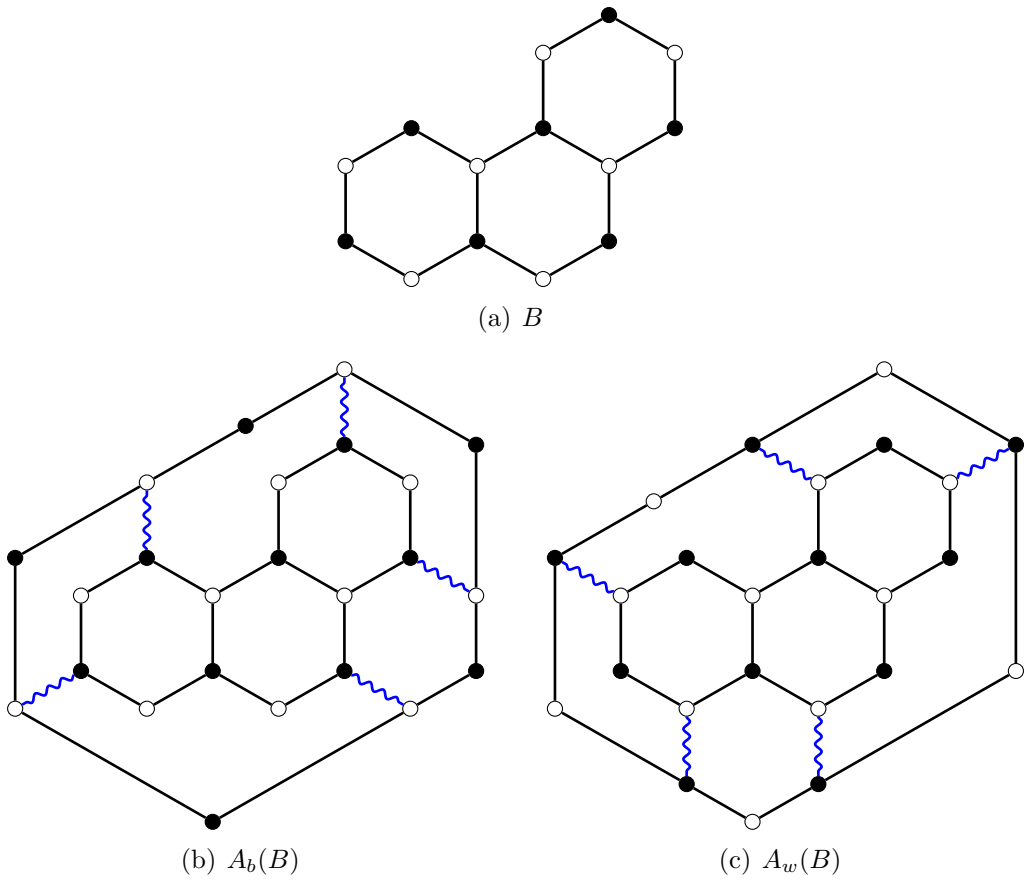
\begin{figure}[!htbp]
\centering
\subfigure[$B$]{
\label{fig:example4sub1}
\input{figure_4.tikz}
}\\
\subfigure[$A_b(B)$]{
\label{fig:example4sub2}
\input{figure_4a.tikz}
}
\subfigure[$A_w(B)$]{
\label{fig:example4sub3}
\input{figure_4b.tikz}
}
\caption{Black and white altans are isomorphic, $A_b(B) \cong A_w(B)$, but come with different orientation.}
\label{fig:example4}
\end{figure}

\begin{figure}[!htbp]
\centering
\subfigure[$B$]{
\label{fig:example4sub1}
\input{figure_4a1.tikz} 
}\\
\subfigure[$A_b(B)$]{
\label{fig:example4sub2}
\input{figure_4aa.tikz} 
}
\subfigure[$A_w(B)$]{
\label{fig:example4sub3}
\input{figure_4ab.tikz} 
}
\caption{Black and white non-isomorphic altans $A_b(B) \ncong A_w(B)$.}
\label{fig:example4a}
\end{figure}
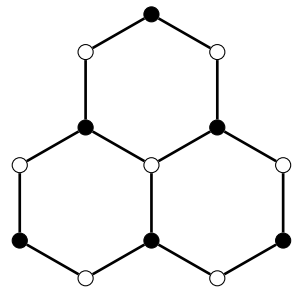
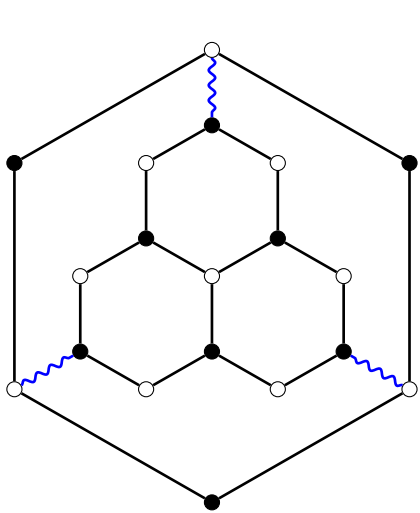
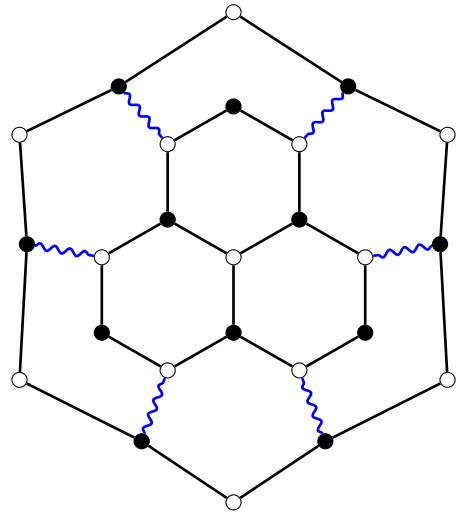

\end{example}

\begin{proposition}
If $G$ is bipartite,
then both $A(G, S_b)$ and $A(G, S_w)$ satisfy conditions of Theorem \ref{thm1}.
\end{proposition}

\begin{proof}
Since all vertices of $S_b$ are coloured black and all vertices of $S_w$ are coloured white, Theorem~\ref{thm1} applies to both black and white altans and the conclusion follows in each case.
\end{proof}

In \cite{gutman2014a,gutman2014} it was shown that $A(G, S)$ has twice the number of Kekul\'{e} structures as $G$. We note that this number is independent of the order in which we choose the vertices of $S$. It is even independent of the choice of $S$ itself. It is not hard to see that $A^n(G)$ at some point becomes similar to a hexagonal nanotube with the original graph $G$ as part of the cap. The structure of the periphery of $A^n(G, S)$ is composed of a ring of $k$ hexagons where $k = |S|$. Here we give the answers to the question of which benzenoids $B$ or more general fullerene patches will give rise to a capped nanotube $A^n(B)$.

There is a special class of benzenoids, called convex benzenoids \cite{cruzgutmanrada2012}. It can be defined in several equivalent ways. Each benzenoid $B$ is uniquely determined by its boundary-edges code $b(B)$ \cite{hansen1996}, see also \cite {kovic2014}, which in turn gives the distances between any two boundary vertices of valence 3. By definition the benzene is assigned the code $6$.

\begin{definition} A benzenoid $B$ is convex if and only if its boundary-edges code $b(B)$ does not contain symbol \tt{1}. \end{definition} 

\begin{example}
The distinction between convex and non-convex benzenoids is visible from their boundary-edges codes. For instance, the boundary-edges code 24334 of the benzenoid on Figure  \ref{fig:example5sub1} is 24334 and does not contain 1. Therefore the
benzenoid is convex while the boundary-edges code for the benzenoid in Figure~\ref{fig:example5sub2} is 144144, containing a 1. Hence the
benzenoid itself is non-convex. 
See Figure \ref{fig:example5}.

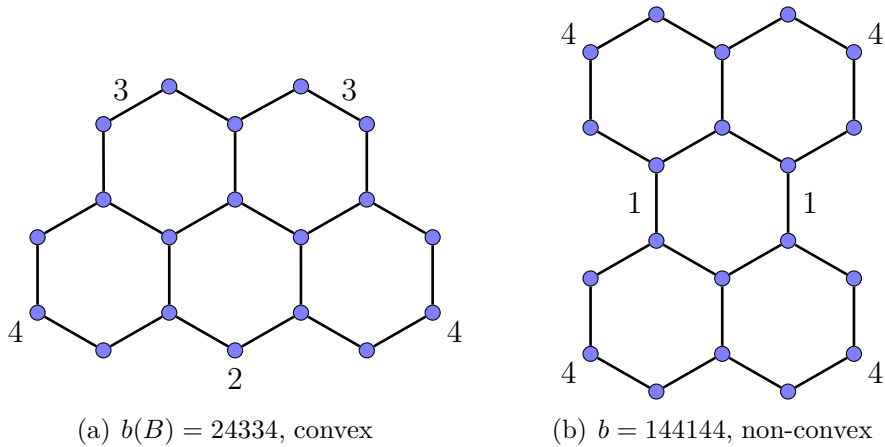
\begin{figure}[!htbp]
\centering
\subfigure[$b(B) = 24334$, convex]{
\label{fig:example5sub1}
\input{figure_5.tikz}
}
\subfigure[$b = 144144$, non-convex]{
\label{fig:example5sub2}
\input{figure_5a.tikz}
}
\caption{The distinction between convex and non-convex benzenoids is visible from their boundary-edges code.}
\label{fig:example5}
\end{figure}
\end{example}

\section{Altans of fullerene- and other patches}

Now we generalize benzenoids to more general planar systems that we call patches. Our definition is more general than the definition of Jack Graver and his co-authors \cite{brinkmann2009,graver2003,gravergravesgraves2014,gravergraves2010,gravesgraves2014,graves2012}. Each patch has a unique boundary-edges code, however unlike benzenoids that are uniquely determined by their boundary-edges codes, the same boundary-edges code may belong to more than one patch or even to a fullerene patch (as defined below).

\begin{definition}
A \emph{patch} $\Pi$ is either a cycle or a 3-valent 2-connected plane (multi)graph with a distinguished outer face whose edges may be arbitrarily subdivided.
\end{definition}

In the following example we depict some benzenoids as patches arising from plane cubic graphs.
\begin{example}
Some plane cubic 2-connected graphs may be expanded to benzenoids; see Figure \ref{fig:example6}.

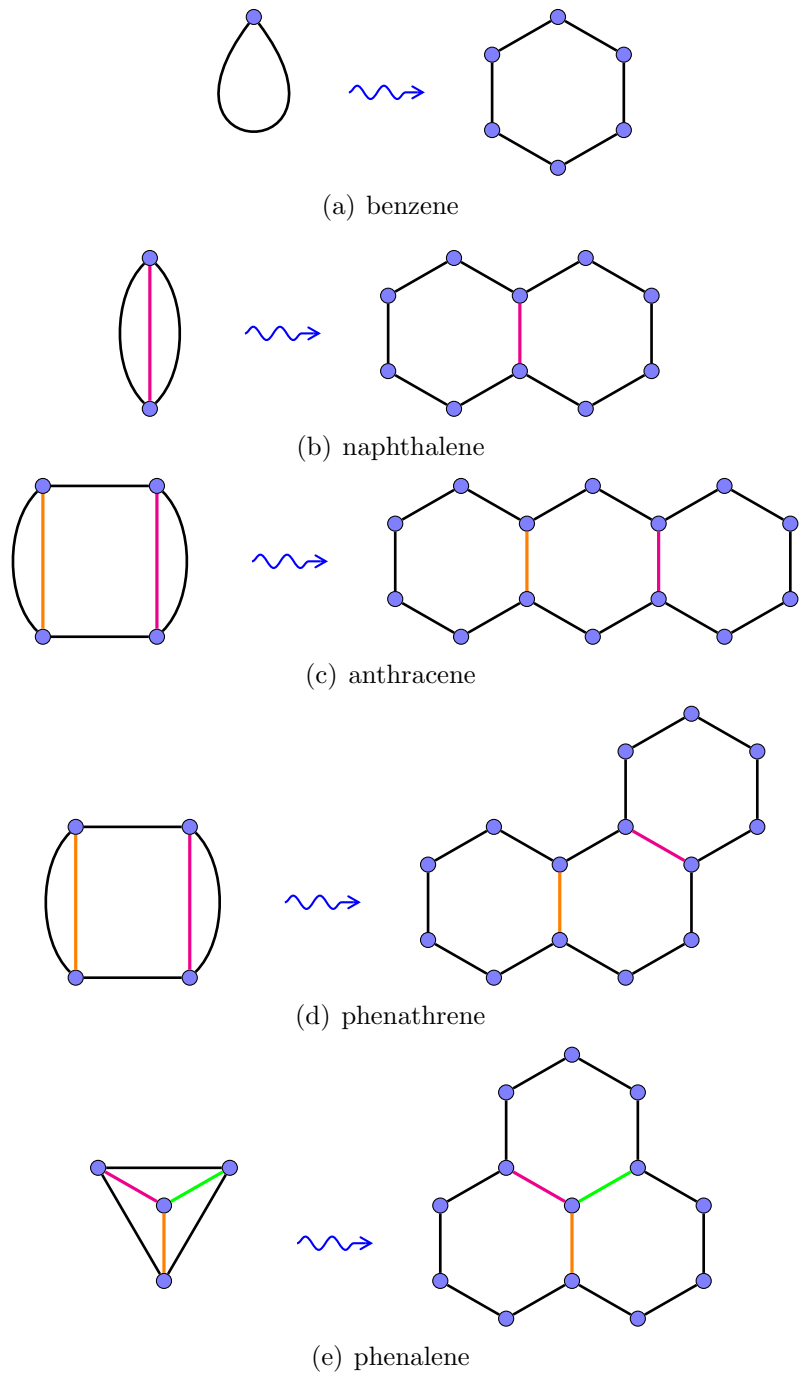
\begin{figure}[!htbp]
\centering
\subfigure[benzene]{
\label{fig:example6sub1}
\input{figure_6.tikz}
}\\
\subfigure[naphthalene]{
\label{fig:example6sub2}
\input{figure_6a.tikz}
}
\subfigure[anthracene]{
\label{fig:example6sub3}
\input{figure_6b.tikz}
}
\subfigure[phenathrene]{
\label{fig:example6sub4}
\input{figure_6c.tikz}
}
\subfigure[phenalene]{
\label{fig:example6sub5}
\input{figure_6d.tikz}
}
\caption{Smallest benzenoids as patches. Benzene is obtained from a loop and the rest are obtained from plane cubic 2-connected graphs.}
\label{fig:example6}
\end{figure}

\end{example}

\begin{example}
The graph in Figure \ref{fig:example7} does not give rise to any patch.

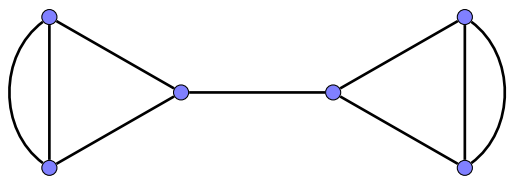
\begin{figure}[!htbp]
\centering
\input{figure_7.tikz}
\caption{A plane cubic graph that is not 2-connected does not have a boundary cycle and is therefore not able to produce any patch.}
\label{fig:example7}
\end{figure}

\end{example}

Clearly a patch is a proper generalization of a benzenoid.

\begin{definition}
A patch $\Phi$ with interior faces pentagons and hexagons is called a \emph{fullerene patch}.
\end{definition}

The above definition coincides with the definition of Graver et al.\ \cite{gravergravesgraves2014,gravergraves2010}.
\begin{definition}
A patch $\Phi$ with interior faces hexagons is called a \emph{helicene patch}.
\end{definition}

\begin{proposition}
The class of helicene patches and helicenes coincide.
\end{proposition}

\begin{proof}
It follows directly from the definitions.
\end{proof}

In particular this means that the following is true:

\begin{proposition}
Each benzenoid is a  helicene patch but there are helicene patches that are not benzenoids.
\end{proposition}

\begin{proof}
The implication follows directly from the definition. It is not an equivalence, since each proper helicene
such as the one from Figure \ref{fig:helicene} is a counter-example to the converse.

\begin{figure}[!htbp]
\centering
\input{figure_9.tikz}
\caption{Not all patches with internal 6-cycles are benzenes. The patch in this figure is a helicene.}
\label{fig:helicene}
\end{figure}
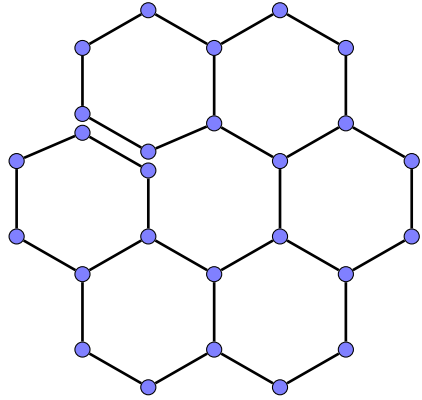

\end{proof}

For a patch $\Pi$, the altan of the patch, $A(\Pi)$, is clearly defined, i.e.\ its peripheral root contains cyclically ordered vertices of valence 2 on the perimeter.

\begin{proposition}
\label{prop:altanpatch}
If $\Pi$ is a patch with $k$ vertices of valence $2$, then $A(\Pi)$ is a patch with $k$ vertices of valence $2$ and
boundary-edges code $2^k$. Furthermore, the faces of $\Pi$ are augmented by a ring of $k$ faces to form the internal faces of $A(\Pi)$.
\end{proposition}

\begin{proof}
This follows directly from the definition of altans.
\end{proof}

A $(k,1)$-nanotube is a ring of $k$ hexagons. We may view is as the black altan of a cycle on $2k$ vertices: $A_b(C_{2k})$. If we glue $s$ such rings one on top of the other we obtain a
$(k,s)$-nanotube.

\begin{corollary}
If $\Pi$ is a patch with boundary code $2^k$ then the ring of faces attached to $\Pi$ when forming $A(\Pi)$ is 
a $(k,1)$-nanotube.
\end{corollary}

\begin{figure}
\centering
\input{figure_12.tikz}
\label{fig:hexaring}
\caption{$(k, 1)$-nanotube}
\end{figure}
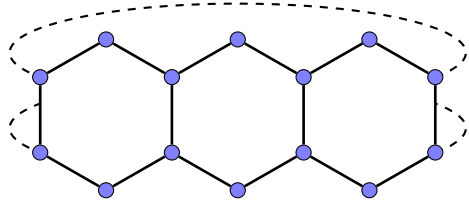

\begin{corollary}
Let $\Pi$ be an arbitrary patch with $k$ vertices of valence 2. The iterated altan $A^n(\Pi)$ is composed of $A(\Pi)$ to which  a $(k,n-1)$-nanotube is attached.
\end{corollary}

\begin{proof}
It follows easily by induction.
\end{proof}

\begin{definition}
A patch $\Pi$ is \emph{convex} if its boundary-edges code contains no $1$.
\end{definition}

\begin{proposition}
Let $\Pi$ be a patch. The ring of new faces of $A(\Pi)$ contains only faces of length $\geq 5$ and $A(\Pi)$ is convex. 
\end{proposition}

\begin{proof}
Each vertex of valence 2 on the peripheral cycle is adjacent to two vertices of valence 3. Therefore each
new face of $A(\Pi)$ contains at least 3 peripheral vertices. It must contain at least two old vertices. Actually it will
contain two consecutive old vertices if and only if they were adjacent boundary vertices of valence 2 of $\Pi$. 
By definition $A(\Pi)$ is convex by Proposition \ref{prop:altanpatch}.
\end{proof}

\begin{proposition}
\label{prop:faces56convex}
Let $\Pi$ be a patch. The ring of new faces of $A(\Pi)$ contains only faces of length 5 or 6 if and only if $\Pi$ is convex. 
\end{proposition}

\begin{proof}
The pentagons are covered already in the prof of previous Proposition.  A face of length $> 6$ appears if and only
if the boundary edges code contains one or more consecutive $1$s.
\end{proof}

\begin{theorem}
A benzenoid $B$ will give rise to a fullerene nano-tube if and only if it is convex.
\end{theorem}
\begin{proof}
It follows by induction from Proposition \ref{prop:faces56convex}.
\end{proof}

\begin{theorem}
If $\Pi$ is a convex fullerene patch with $p$ pentagons and boundary-edges code:
$b(\Pi) = (2 + a_1)(2 + a_2) \ldots (2 + a_k)$, $a_i \geq 0$, $d = a_1 + a_2 + \ldots + a_k$. Then $d+p = 6$
and $A(\Pi)$ is a fullerene patch with edges boundary code $2^k$ and 6 pentagons.
\end{theorem}
\begin{proof}
This result may be viewed as a consequence of the fact that a nanotube may be capped only by a fullerene patch
having 6 pentagons. All action takes place at the initial step passing from $\Pi$ to $A(\Pi)$ since the width of the tube
cones not change after eventual iterated applications of altan operation. 
\end{proof}

\section{Kekul\'{e} structures of iterated altans}
We have seen that $A^n(G)$ behaves essentially like a capped nanotube. In the first step from $G$ to $A(G)$ all
irregularities happen. After that each $A^{n+1}(G)$ is obtained from $A^n(G)$ by attaching a ring of hexagons to its periphery.

Gutman \cite{gutman2014a} proved the following:

\begin{lemma}[Gutman, 2014]
Let $G$ be a graph having $K = K(G)$ Kekul\'{e} structures. Then any of its altans $A(G)$ has $2K$ i.e.\ twice the number of 
Kekul\'{e} structures.
\end{lemma}

We may apply it to the iterated altans:

\begin{theorem}
The number of Kekul\'{e} structures in the $n$-th iterated altan of $G$ is $2^n K(G)$.
\end{theorem}

\begin{proof}
From Gutman's Lemma by induction.
\end{proof}

This result has many interesting consequences. The first one is confirmation of the result of Sachs et al.\ \cite{sachs1996}
that the number of Kekul\'{e} structures of a $(k,s)$-nanotube is independent of $k$ and is equal to $2^{s+1}$. More generally we may compute the number of Kekul\'{e} structures of any patch extended by a nanotube.

\begin{corollary}
Let $\Pi$ be a patch then $K(A^n(\Pi)) = 2^nK(\Pi)$.
\end{corollary} 

In particular this means that a nanotube capped by six pentagons (half of a dodecahedron) has no Kekul\'{e} structures, while the nanotube capped by half of the buckyball and its 11 Kekul\'{e} structures gives rise to the total of $11 \times 2^n$  Kekul\'{e} structures. We should mention that our nanotubes correspond to a very special \emph{untwisted} case of much more general nanotubes, alias \emph{tubules} considered in the 73 page paper \cite{sachs1996}.

\begin{figure}
\centering
\input{figure_11.tikz}
\label{fig:sfoldhexaring}
\caption{$(k, s)$-nanotube}
\end{figure}
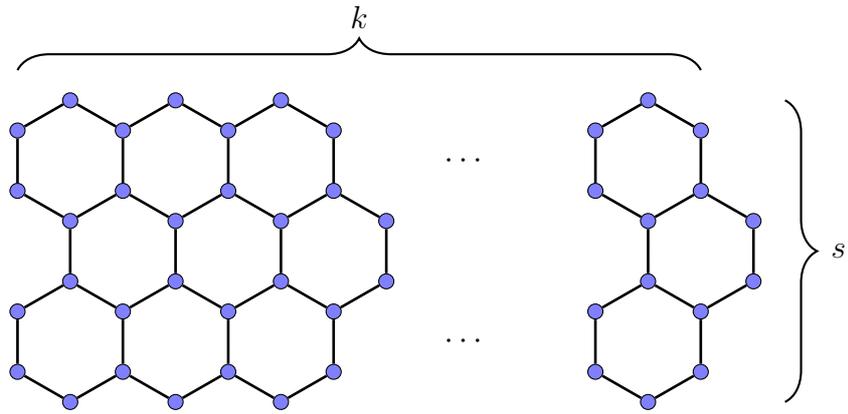

\begin{figure}[!htbp]
\centering
\subfigure[]{
\label{fig:kapica56a}
\input{kapica_5_6.tikz}
}
\subfigure[]{
\label{fig:kapica56b}
\input{kapica_5_6_altansemi.tikz}
}\\
\subfigure[]{
\label{fig:kapica56c}
\input{kapica_5_6_altan.tikz}
}
\subfigure[]{
\label{fig:kapica56d}
\input{kapica_5_6_altan2semi.tikz}
}
\label{fig:kapica56}
\caption{A cap $\Pi$ with a single pentagon (see (a), (b)) that turns into a bucky-ball dome $A(\Pi)$ with six pentagons (see (c), (d)) after the first altan and to a longer capped nanotube after any additional altan operation.}
\end{figure}
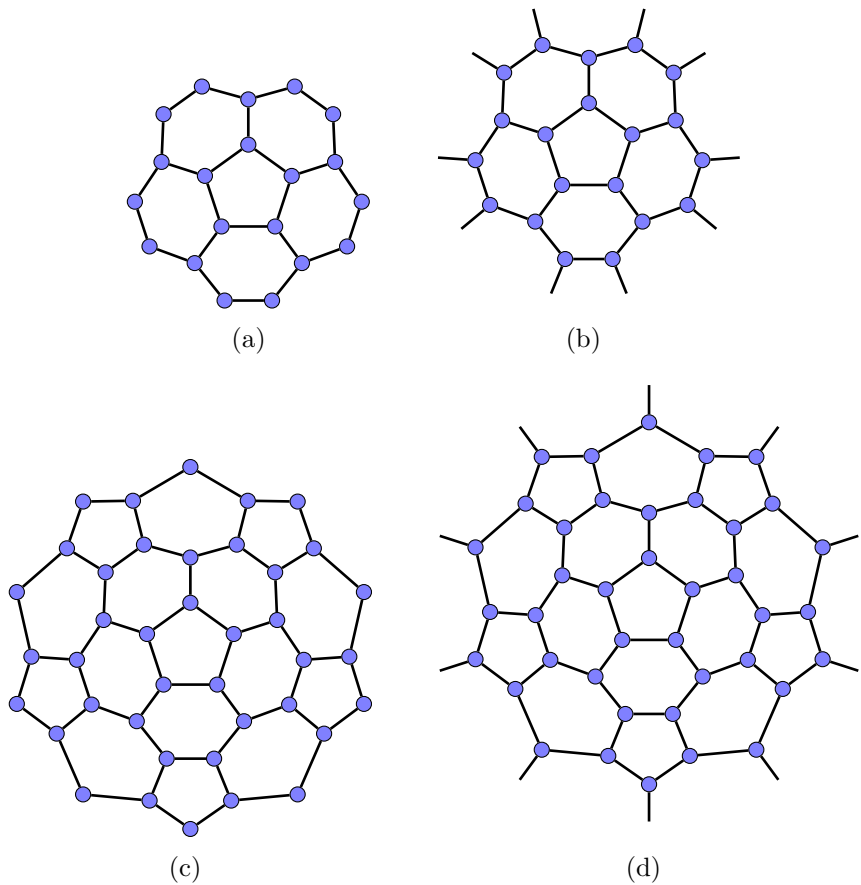

\vspace{\baselineskip}
\noindent
{\bf Acknowledgements:} The authors would like to thank Patrick W.\ Fowler, Jack Graver and Milan Randi\'{c} for useful conversation in preparation of this paper. Research supported in part by the ARRS Grant P1-0294. This material is based upon work supported
by the Air Force Office of Scientific Research, Air Force Materiel
Command, USAF under Award No.\ FA9550-14-1-0096.

\bibliographystyle{amsplain-nodash}
\bibliography{iterated_altans_nb}

\end{document}

%% file: figure_1.tikz
\begin{tikzpicture}[scale=0.5]
\tikzstyle{every node} = [inner sep=2, draw, circle, fill=blue!50]
\tikzstyle{edge} = [draw, line width=1.0]
\pgfmathsetmacro{\k}{360 / 6};
\foreach \i in {0,...,5} {
  \node[label={-90+\k*\i}:$\i$] (u_\i) at (-90+\k*\i:2.0) {};
}
\foreach \i in {0,...,5} {
  \pgfmathtruncatemacro{\j}{mod(\i+1, 6)};
  \draw[edge] (u_\i) -- (u_\j);
}
\end{tikzpicture}

%% file: figure_1a.tikz
\begin{tikzpicture}[scale=0.5]
\tikzstyle{every node} = [inner sep=2, draw, circle, fill=blue!50]
\tikzstyle{edge} = [draw, line width=1.0]
\pgfmathsetmacro{\k}{360 / 6};
\foreach \i in {0,...,5} {
   \pgfmathtruncatemacro{\j}{\i+6};
   \pgfmathtruncatemacro{\jj}{\i+12};
 \node[label=60-90+\k*\i:\i] (u_\i) at (-90+\k*\i:2.0) {};
  \node[label=-90+\k*\i:\j] (v_\i) at (-90+\k*\i:3.7) {};
  \draw[edge] (u_\i) -- (v_\i);
  \node[label=-90+\k*\i+360/12:\jj] (w_\i) at (-90+\k*\i+360/12:4.5) {};
}
\foreach \i in {0,...,5} {
  \pgfmathtruncatemacro{\j}{mod(\i+1, 6)};
  \draw[edge] (u_\i) -- (u_\j);
  \draw[edge] (v_\i) -- (w_\i) -- (v_\j);
}
\end{tikzpicture}

%% file: figure_2.tikz
\begin{tikzpicture}[scale=0.5]
\tikzstyle{every node} = [inner sep=2, draw, circle, fill=blue!50]
\tikzstyle{edge} = [draw, line width=1.0]
\node (s) at (0, 0) {};
\pgfmathsetmacro{\k}{360 / 6};
\foreach \i in {0,...,2} {
  \node (u_\i) at (90+2*\k*\i:2.0) {};
  \draw[edge] (u_\i) -- (s);
}
\node[label=30:$b$] (v_0) at (30:2.0) {};
\node[label={30+2*\k*1}:$c$] (v_1) at (30+2*\k*1:2.0) {};
\node[label={30+2*\k*2}:$a$] (v_2) at (30+2*\k*2:2.0) {};
\foreach \i in {0,...,2} {
  \pgfmathtruncatemacro{\j}{mod(\i+1, 3)};
  \draw[edge] (v_\i) -- (u_\i) -- (v_\j);
}

\end{tikzpicture}

%% file: figure_2a.tikz
\begin{tikzpicture}[scale=0.5]
\tikzstyle{every node} = [inner sep=2, draw, circle, fill=blue!50]
\tikzstyle{edge} = [draw, line width=1.0]
\node (s) at (0, 0) {};
\pgfmathsetmacro{\k}{360 / 6};
\foreach \i in {0,...,2} {
  \node (u_\i) at (90+2*\k*\i:2.0) {};
  \draw[edge] (u_\i) -- (s);
}
\node (v_0) at (30:2.0) {};
\node (v_1) at (30+2*\k*1:2.0) {};
\node (v_2) at (30+2*\k*2:2.0) {};
\foreach \i in {0,...,2} {
  \pgfmathtruncatemacro{\j}{mod(\i+1, 3)};
  \draw[edge] (v_\i) -- (u_\i) -- (v_\j);
}
  \node[label=90:$b'$] (a_0) at (90+2*\k*0:4.0) {};
  \node[label=180+30:$c'$] (a_1) at (90+2*\k*1:4.0) {};
  \node[label=-30:$a'$] (a_2) at (90+2*\k*2:4.0) {};
\foreach \i in {0,...,2} {
  \node (b_\i) at (30+2*\k*\i:4.0) {};
}
\foreach \i in {0,...,2} {
  \pgfmathtruncatemacro{\j}{mod(\i+1, 3)};
  \draw[edge] (b_\j) -- (a_\i) -- (b_\i) -- (v_\i);
}
\end{tikzpicture}

%% file: figure_3.tikz
\begin{tikzpicture}[scale=0.5]
\tikzstyle{every node} = [inner sep=2, draw, circle, fill=white]
\tikzstyle{edge} = [draw, line width=1.0]
\pgfmathsetmacro{\k}{360 / 6};
\foreach \i in {0,...,2} {
  \node (u_\i) at (150+2*\k*\i:2.0) {};
  \node[fill=black] (v_\i) at (90+2*\k*\i:2.0) {};
}
\foreach \i in {0,...,2} {
  \pgfmathtruncatemacro{\j}{mod(\i+1, 3)};
  \draw[edge] (v_\i) -- (u_\i) -- (v_\j);
}
\end{tikzpicture}

%% file: figure_3a.tikz
\begin{tikzpicture}[scale=0.5]
\tikzstyle{every node} = [inner sep=2, draw, circle, fill=white]
\tikzstyle{edge} = [draw, line width=1.0]
\pgfmathsetmacro{\k}{360 / 6};
\foreach \i in {0,...,2} {
  \node (u_\i) at (150+2*\k*\i:2.0) {};
  \node[fill=black] (v_\i) at (90+2*\k*\i:2.0) {};
}
\foreach \i in {0,...,2} {
  \pgfmathtruncatemacro{\j}{mod(\i+1, 3)};
  \draw[edge] (v_\i) -- (u_\i) -- (v_\j);
}
\foreach \i in {0,...,2} {
  \node[fill=black] (a_\i) at (150+2*\k*\i:4.0) {};
  \node (b_\i) at (90+2*\k*\i:4.0) {};
}
\foreach \i in {0,...,2} {
  \pgfmathtruncatemacro{\j}{mod(\i+1, 3)};
  \draw[edge] (b_\j) -- (a_\i) -- (b_\i) -- (v_\i);
}
\end{tikzpicture}

%% file: figure_4.tikz
\usetikzlibrary{calc}
\begin{tikzpicture}[scale=0.5]
\tikzstyle{every node} = [inner sep=2, draw, circle, fill=white]
\tikzstyle{edge} = [draw, line width=1.0]
\pgfmathsetmacro{\k}{360 / 6};
\foreach \i in {0,...,2} {
  \node (u_\i) at (150+2*\k*\i:2.0) {};
  \node[fill=black] (v_\i) at (90+2*\k*\i:2.0) {};
}
\foreach \i in {0,...,2} {
  \pgfmathtruncatemacro{\j}{mod(\i+1, 3)};
  \draw[edge] (v_\i) -- (u_\i) -- (v_\j);
}
\foreach \i in {0,...,1} {
  \node (u2_\i) at ($sqrt(3)*(2, 0) + (-90+2*\k*\i:2.0) $) {};
  \node[fill=black] (v2_\i) at ($sqrt(3)*(2, 0) + (-30+2*\k*\i:2.0)$) {};
}
\draw[edge] (v_2) -- (u2_0) -- (v2_0) -- (u2_1) -- (v2_1) -- (u_2);
\foreach \i in {0,...,1} {
  \node (u3_\i) at ($(0, 3) + sqrt(3)*(3, 0) + (30+2*\k*\i:2.0) $) {};
 \node[fill=black] (v3_\i) at ($(0, 3) + sqrt(3)*(3, 0) + (-30+2*\k*\i:2.0)$) {};
}
\draw[edge] (u2_1) -- (v3_0) -- (u3_0) -- (v3_1) -- (u3_1) -- (v2_1);
\end{tikzpicture}

%% file: figure_4a.tikz
\usetikzlibrary{calc}
\usetikzlibrary{decorations.pathmorphing}
\begin{tikzpicture}[scale=0.5]
\tikzstyle{every node} = [inner sep=2, draw, circle, fill=white]
\tikzstyle{edge} = [draw, line width=1.0]
\tikzstyle{spoke} = [draw, line width=1.0, decorate, decoration={snake, segment length=2mm, amplitude=1.2pt}, color=blue]
\pgfmathsetmacro{\k}{360 / 6};
\foreach \i in {0,...,2} {
  \node (u_\i) at (150+2*\k*\i:2.0) {};
  \node[fill=black] (v_\i) at (90+2*\k*\i:2.0) {};
}
\foreach \i in {0,...,2} {
  \pgfmathtruncatemacro{\j}{mod(\i+1, 3)};
  \draw[edge] (v_\i) -- (u_\i) -- (v_\j);
}
\foreach \i in {0,...,1} {
  \node (u2_\i) at ($sqrt(3)*(2, 0) + (-90+2*\k*\i:2.0) $) {};
  \node[fill=black] (v2_\i) at ($sqrt(3)*(2, 0) + (-30+2*\k*\i:2.0)$) {};
}
\draw[edge] (v_2) -- (u2_0) -- (v2_0) -- (u2_1) -- (v2_1) -- (u_2);
\foreach \i in {0,...,1} {
  \node (u3_\i) at ($(0, 3) + sqrt(3)*(3, 0) + (30+2*\k*\i:2.0) $) {};
 \node[fill=black] (v3_\i) at ($(0, 3) + sqrt(3)*(3, 0) + (-30+2*\k*\i:2.0)$) {};
}
\draw[edge] (u2_1) -- (v3_0) -- (u3_0) -- (v3_1) -- (u3_1) -- (v2_1);

\node (b_0) at ($(v3_0) + (-30:2)$) {}; \draw[spoke] (v3_0) -- (b_0);
\node (b_1) at ($(v3_1) + (90:2)$) {}; \draw[spoke] (v3_1) -- (b_1);
\node (b_2) at ($(v_0) + (90:2)$) {}; \draw[spoke] (v_0) -- (b_2);
\node (b_3) at ($(v_1) + (210:2)$) {}; \draw[spoke] (v_1) -- (b_3);
\node (b_4) at ($(v2_0) + (-30:2)$) {}; \draw[spoke] (v2_0) -- (b_4);

\node[fill=black] (x_0) at ($(u3_0) + (30:2)$) {};
\draw[edge] (b_1) -- (x_0) -- (b_0);
\node[fill=black] (x_1) at ($(v_2) + (-90:4)$) {};
\draw[edge] (b_3) -- (x_1) -- (b_4);
\node[fill=black] (x_2) at ($(u_0) + (150:2)$) {};
\draw[edge] (b_3) -- (x_2) -- (b_2);

\node[fill=black] (x_3) at ($0.5*(b_1) + 0.5*(b_2)$) {};
\draw[edge] (b_2) -- (x_3) -- (b_1);

\node[fill=black] (x_4) at ($(u2_1) + (-30:4)$) {};
\draw[edge] (b_0) -- (x_4) -- (b_4);
\end{tikzpicture}

%% file: figure_4b.tikz
\usetikzlibrary{calc}
\usetikzlibrary{decorations.pathmorphing}
\begin{tikzpicture}[scale=0.5]
\tikzstyle{every node} = [inner sep=2, draw, circle, fill=white]
\tikzstyle{edge} = [draw, line width=1.0]
\tikzstyle{spoke} = [draw, line width=1.0, decorate, decoration={snake, segment length=2mm, amplitude=1.2pt}, color=blue]
\pgfmathsetmacro{\k}{360 / 6};
\foreach \i in {0,...,2} {
  \node (u_\i) at (150+2*\k*\i:2.0) {};
  \node[fill=black] (v_\i) at (90+2*\k*\i:2.0) {};
}
\foreach \i in {0,...,2} {
  \pgfmathtruncatemacro{\j}{mod(\i+1, 3)};
  \draw[edge] (v_\i) -- (u_\i) -- (v_\j);
}
\foreach \i in {0,...,1} {
  \node (u2_\i) at ($sqrt(3)*(2, 0) + (-90+2*\k*\i:2.0) $) {};
  \node[fill=black] (v2_\i) at ($sqrt(3)*(2, 0) + (-30+2*\k*\i:2.0)$) {};
}
\draw[edge] (v_2) -- (u2_0) -- (v2_0) -- (u2_1) -- (v2_1) -- (u_2);
\foreach \i in {0,...,1} {
  \node (u3_\i) at ($(0, 3) + sqrt(3)*(3, 0) + (30+2*\k*\i:2.0) $) {};
 \node[fill=black] (v3_\i) at ($(0, 3) + sqrt(3)*(3, 0) + (-30+2*\k*\i:2.0)$) {};
}
\draw[edge] (u2_1) -- (v3_0) -- (u3_0) -- (v3_1) -- (u3_1) -- (v2_1);

\node[fill=black] (b_0) at ($(u3_0) + (30:2)$) {}; \draw[spoke] (u3_0) -- (b_0);
\node[fill=black] (b_1) at ($(u3_1) + (150:2)$) {}; \draw[spoke] (u3_1) -- (b_1);
\node[fill=black] (b_2) at ($(u_0) + (150:2)$) {}; \draw[spoke] (u_0) -- (b_2);
\node[fill=black] (b_3) at ($(u_1) + (-90:2)$) {}; \draw[spoke] (u_1) -- (b_3);
\node[fill=black] (b_4) at ($(u2_0) + (-90:2)$) {}; \draw[spoke] (u2_0) -- (b_4);

\node (x_0) at ($(v3_1) + (90:2)$) {};
\draw[edge] (b_1) -- (x_0) -- (b_0);
\node (x_1) at ($(v_2) + (-90:4)$) {};
\draw[edge] (b_3) -- (x_1) -- (b_4);
\node (x_2) at ($(v_1) + (210:2)$) {};
\draw[edge] (b_3) -- (x_2) -- (b_2);
\node (x_3) at ($0.5*(b_1) + 0.5*(b_2)$) {};
\draw[edge] (b_2) -- (x_3) -- (b_1);
\node (x_4) at ($(u2_1) + (-30:4)$) {};
\draw[edge] (b_0) -- (x_4) -- (b_4);
\end{tikzpicture}

%% file: figure_4a1.tikz
\usetikzlibrary{calc}
\begin{tikzpicture}[scale=0.5]
\tikzstyle{every node} = [inner sep=2, draw, circle, fill=white]
\tikzstyle{edge} = [draw, line width=1.0]
\pgfmathsetmacro{\k}{360 / 6};
\foreach \i in {0,...,2} {
  \node (u_\i) at (150+2*\k*\i:2.0) {};
  \node[fill=black] (v_\i) at (90+2*\k*\i:2.0) {};
}
\foreach \i in {0,...,2} {
  \pgfmathtruncatemacro{\j}{mod(\i+1, 3)};
  \draw[edge] (v_\i) -- (u_\i) -- (v_\j);
}
\foreach \i in {0,...,1} {
  \node (u2_\i) at ($sqrt(3)*(2, 0) + (-90+2*\k*\i:2.0) $) {};
  \node[fill=black] (v2_\i) at ($sqrt(3)*(2, 0) + (-30+2*\k*\i:2.0)$) {};
}
\draw[edge] (v_2) -- (u2_0) -- (v2_0) -- (u2_1) -- (v2_1) -- (u_2);
\foreach \i in {0,...,1} {
  \node (u3_\i) at ($(0, 3) + sqrt(3)*(1, 0) + (30+2*\k*\i:2.0) $) {};
}
 \node[fill=black] (v3_0) at ($(0, 3) + sqrt(3)*(1, 0) + (90:2.0)$) {};
\draw[edge] (v2_1) -- (u3_0) -- (v3_0) -- (u3_1) -- (v_0);
\end{tikzpicture}

%% file: figure_4aa.tikz
\usetikzlibrary{calc}
\usetikzlibrary{decorations.pathmorphing}
\begin{tikzpicture}[scale=0.5]
\tikzstyle{every node} = [inner sep=2, draw, circle, fill=white]
\tikzstyle{edge} = [draw, line width=1.0]
\tikzstyle{spoke} = [draw, line width=1.0, decorate, decoration={snake, segment length=2mm, amplitude=1.2pt}, color=blue]
\pgfmathsetmacro{\k}{360 / 6};
\foreach \i in {0,...,2} {
  \node (u_\i) at (150+2*\k*\i:2.0) {};
  \node[fill=black] (v_\i) at (90+2*\k*\i:2.0) {};
}
\foreach \i in {0,...,2} {
  \pgfmathtruncatemacro{\j}{mod(\i+1, 3)};
  \draw[edge] (v_\i) -- (u_\i) -- (v_\j);
}
\foreach \i in {0,...,1} {
  \node (u2_\i) at ($sqrt(3)*(2, 0) + (-90+2*\k*\i:2.0) $) {};
  \node[fill=black] (v2_\i) at ($sqrt(3)*(2, 0) + (-30+2*\k*\i:2.0)$) {};
}
\draw[edge] (v_2) -- (u2_0) -- (v2_0) -- (u2_1) -- (v2_1) -- (u_2);
\foreach \i in {0,...,1} {
  \node (u3_\i) at ($(0, 3) + sqrt(3)*(1, 0) + (30+2*\k*\i:2.0) $) {};
}
 \node[fill=black] (v3_0) at ($(0, 3) + sqrt(3)*(1, 0) + (90:2.0)$) {};
\draw[edge] (v2_1) -- (u3_0) -- (v3_0) -- (u3_1) -- (v_0);
\node (b_0) at ($(v3_0) + (0, 2)$) {};
\node (b_1) at ($(v2_0) + (-30:2)$) {};
\node (b_2) at ($(v_1) + (180+30:2)$) {};
\node[fill=black] (a_0) at ($(v_0) + (150:4)$) {};
\node[fill=black] (a_1) at ($(v_2) + (-90:4)$) {};
\node[fill=black] (a_2) at ($(v2_1) + (30:4)$) {};
\draw[edge] (b_0) -- (a_0) -- (b_2) -- (a_1) -- (b_1) -- (a_2) -- (b_0);
\draw[spoke] (b_0) -- (v3_0);
\draw[spoke] (b_1) -- (v2_0);
\draw[spoke] (b_2) -- (v_1);
\end{tikzpicture}

%% file: figure_4ab.tikz
\usetikzlibrary{calc}
\usetikzlibrary{decorations.pathmorphing}
\begin{tikzpicture}[scale=0.5]
\tikzstyle{every node} = [inner sep=2, draw, circle, fill=white]
\tikzstyle{edge} = [draw, line width=1.0]
\tikzstyle{spoke} = [draw, line width=1.0, decorate, decoration={snake, segment length=2mm, amplitude=1.2pt}, color=blue]
\pgfmathsetmacro{\k}{360 / 6};
\foreach \i in {0,...,2} {
  \node (u_\i) at (150+2*\k*\i:2.0) {};
  \node[fill=black] (v_\i) at (90+2*\k*\i:2.0) {};
}
\foreach \i in {0,...,2} {
  \pgfmathtruncatemacro{\j}{mod(\i+1, 3)};
  \draw[edge] (v_\i) -- (u_\i) -- (v_\j);
}
\foreach \i in {0,...,1} {
  \node (u2_\i) at ($sqrt(3)*(2, 0) + (-90+2*\k*\i:2.0) $) {};
  \node[fill=black] (v2_\i) at ($sqrt(3)*(2, 0) + (-30+2*\k*\i:2.0)$) {};
}
\draw[edge] (v_2) -- (u2_0) -- (v2_0) -- (u2_1) -- (v2_1) -- (u_2);
\foreach \i in {0,...,1} {
  \node (u3_\i) at ($(0, 3) + sqrt(3)*(1, 0) + (30+2*\k*\i:2.0) $) {};
}
 \node[fill=black] (v3_0) at ($(0, 3) + sqrt(3)*(1, 0) + (90:2.0)$) {};
\draw[edge] (v2_1) -- (u3_0) -- (v3_0) -- (u3_1) -- (v_0);
\node[fill=black] (b_0) at ($(u3_0) + (50:2)$) {}; \draw[spoke] (b_0) -- (u3_0);
\node[fill=black] (b_1) at ($(u3_1) + (130:2)$) {}; \draw[spoke] (b_1) -- (u3_1);
\node[fill=black] (b_2) at ($(u_1) + (-110:2)$) {}; \draw[spoke] (b_2) -- (u_1);
\node[fill=black] (b_3) at ($(u2_1) + (10:2)$) {}; \draw[spoke] (b_3) -- (u2_1);
\node[fill=black] (b_4) at ($(u2_0) + (-70:2)$) {}; \draw[spoke] (b_4) -- (u2_0);
\node[fill=black] (b_5) at ($(u_0) + (170:2)$) {}; \draw[spoke] (b_5) -- (u_0);
\node (a_0) at ($(v3_0) + (90:2.5)$) {};
\node (a_1) at ($(v_2) + (-90:4.5)$) {};
\node (a_2) at ($(v_1) + (210:2.5)$) {};
\node (a_3) at ($(v_0) + (150:4.5)$) {};
\node (a_4) at ($(v2_0) + (-30:2.5)$) {};
\node (a_5) at ($(v2_1) + (30:4.5)$) {};
\draw[edge] (b_0) -- (a_0) -- (b_1) -- (a_3) -- (b_5) -- (a_2) -- (b_2) -- (a_1) -- (b_4) -- (a_4) -- (b_3) -- (a_5) -- (b_0);
\end{tikzpicture}

%% file: figure_5.tikz
\usetikzlibrary{calc}
\usetikzlibrary{decorations.pathmorphing}
\usetikzlibrary{arrows}
\begin{tikzpicture}[scale=0.5,>=angle 60]
\tikzstyle{every node} = [inner sep=2, draw, circle, fill=blue!50]
\tikzstyle{edge} = [draw, line width=1.0]
\pgfmathsetmacro{\k}{360 / 6};

\foreach \i in {0,...,4} {
  \node (u_\i) at (-90+\k*\i:2.0) {};
}
  \node[label=210:4] (u_5) at (-90+\k*5:2.0) {};
\draw[edge] (u_0) -- (u_1) -- (u_2) -- (u_3) -- (u_4) -- (u_5) -- (u_0);

\foreach \i in {1,...,3} {
  \node (d_\i) at ($sqrt(3)*(2,0) + (-90+\k*\i:2.0)$) {};
}
  \node[label=-90:$2$] (d_0) at ($sqrt(3)*(2,0) + (-90+\k*0:2.0)$) {};

\draw[edge] (u_1) -- (d_0) -- (d_1) -- (d_2) -- (d_3) -- (u_2);

\foreach \i in {2,...,3} {
  \node (e_\i) at ($sqrt(3)*(4,0) + (-90+\k*\i:2.0)$) {};
}
 \node (e_0) at ($sqrt(3)*(4,0) + (-90+\k*0:2.0)$) {};
 \node[label=-30:$4$] (e_1) at ($sqrt(3)*(4,0) + (-90+\k*1:2.0)$) {};

\draw[edge] (d_1) -- (e_0) -- (e_1) -- (e_2) -- (e_3) -- (d_2);

\foreach \i in {0,...,2} {
  \node (a_\i) at ($sqrt(3)*(3,0) + (0, 3) + (30+\k*\i:2.0)$) {};
}
\draw[edge] (e_3) -- (a_0) -- node[above right,fill=none,draw=none] {$3$} ++ (a_1) -- (a_2) -- (d_3);

\foreach \i in {0,...,1} {
  \node (b_\i) at ($sqrt(3)*(1,0) + (0, 3) + (90+\k*\i:2.0)$) {};
}
\draw[edge] (a_2) -- (b_0) -- node[above left,fill=none,draw=none] {$3$} ++ (b_1) -- (u_3);

\end{tikzpicture}

%% file: figure_5a.tikz
\usetikzlibrary{calc}
\usetikzlibrary{decorations.pathmorphing}
\usetikzlibrary{arrows}
\begin{tikzpicture}[scale=0.5,>=angle 60]
\tikzstyle{every node} = [inner sep=2, draw, circle, fill=blue!50]
\tikzstyle{edge} = [draw, line width=1.0]
\pgfmathsetmacro{\k}{360 / 6};

\foreach \i in {0,...,4} {
  \node (u_\i) at (-90+\k*\i:2.0) {};
}
\node[label=210:4] (u_5) at (-90+\k*5:2.0) {};
\draw[edge] (u_0) -- (u_1) -- (u_2) -- (u_3) -- (u_4) -- (u_5) -- (u_0);

\foreach \i in {2,...,3} {
  \node (d_\i) at ($sqrt(3)*(2,0) + (-90+\k*\i:2.0)$) {};
}
 \node (d_0) at ($sqrt(3)*(2,0) + (-90+\k*0:2.0)$) {};
\node[label=-30:$4$] (d_1) at ($sqrt(3)*(2,0) + (-90+\k*1:2.0)$) {};
\draw[edge] (u_1) -- (d_0) -- (d_1) -- (d_2) -- (d_3) -- (u_2);

\foreach \i in {0,...,3} {
  \node (u2_\i) at ($(0,6) + (-90+\k*\i:2.0)$) {};
}
\node[label=150:4] (u2_4) at ($(0, 6) + (-90+\k*4:2.0)$) {};
\node (u2_5) at ($(0, 6) + (-90+\k*5:2.0)$) {};
\draw[edge] (u2_0) -- (u2_1) -- (u2_2) -- (u2_3) -- (u2_4) -- (u2_5) -- (u2_0);

\foreach \i in {0,1,3} {
  \node (d2_\i) at ($sqrt(3)*(2,0) + (0, 6) + (-90+\k*\i:2.0)$) {};
}
\node[label=30:$4$] (d2_2) at ($sqrt(3)*(2,0) +(0, 6) + (-90+\k*2:2.0)$) {};
\draw[edge] (u2_1) -- (d2_0) -- (d2_1) -- (d2_2) -- (d2_3) -- (u2_2);

\draw[edge] (u_3) -- node[left,fill=none,draw=none] {$1$} ++ (u2_0);
\draw[edge] (d_3) -- node[right,fill=none,draw=none] {$1$} ++ (d2_0);

\end{tikzpicture}

%% file: figure_6.tikz
\begin{tikzpicture}[scale=0.5,>=angle 60]
\tikzstyle{every node} = [inner sep=2, draw, circle, fill=blue!50]
\tikzstyle{edge} = [draw, line width=1.0]
\pgfmathsetmacro{\k}{360 / 6};
\node (a) at (-8, 2) {};
\draw[edge] (a) .. controls ($(a) + (-3, -4)$) and ($(a) + (3, -4)$) .. (a);
\path [thick,draw=blue,style={decorate, decoration=snake}] (-5.5,0) -- (-3.7,0);
\path [thick,draw=blue,->] (-3.7,0) -- (-3.5,0);
\foreach \i in {0,...,5} {
  \node (u_\i) at (-90+\k*\i:2.0) {};
}
\foreach \i in {0,...,5} {
  \pgfmathtruncatemacro{\j}{mod(\i+1, 6)};
  \draw[edge] (u_\i) -- (u_\j);
}
\end{tikzpicture}

%% file: figure_6a.tikz
\begin{tikzpicture}[scale=0.5,>=angle 60]
\tikzstyle{every node} = [inner sep=2, draw, circle, fill=blue!50]
\tikzstyle{edge} = [draw, line width=1.0]
\pgfmathsetmacro{\k}{360 / 6};
\node (a) at (-8, 2) {};
\node (b) at (-8, -2) {};

\draw[edge] (a) .. controls (-9,1) and (-9,-1) .. (b);
\draw[edge] (a) .. controls (-7,1) and (-7,-1) .. (b);
\draw[edge,color=magenta,very thick] (a) -- (b);

\path [thick,draw=blue,style={decorate, decoration=snake}] (-5.5,0) -- (-3.7,0);
\path [thick,draw=blue,->] (-3.7,0) -- (-3.5,0);
\foreach \i in {0,...,5} {
  \node (u_\i) at (-90+\k*\i:2.0) {};
}
\foreach \i in {0,...,3} {
  \node (d_\i) at ($sqrt(3)*(2,0) + (-90+\k*\i:2.0)$) {};
}
\draw[edge] (u_1) -- (d_0) -- (d_1) -- (d_2) -- (d_3) -- (u_2);
\draw[edge] (u_0) -- (u_1);
\draw[edge,color=magenta,very thick] (u_1) -- (u_2);
\draw[edge] (u_2) -- (u_3) -- (u_4) -- (u_5) -- (u_0);
\end{tikzpicture}

%% file: figure_6b.tikz
\usetikzlibrary{calc}
\usetikzlibrary{decorations.pathmorphing}
\usetikzlibrary{arrows}
\begin{tikzpicture}[scale=0.5,>=angle 60]
\tikzstyle{every node} = [inner sep=2, draw, circle, fill=blue!50]
\tikzstyle{edge} = [draw, line width=1.0]
\pgfmathsetmacro{\k}{360 / 6};
\node (a) at (-8, 2) {};
\node (b) at (-8, -2) {};
\node (a2) at (-11, 2) {};
\node (b2) at (-11, -2) {};

\draw[edge] (a2) .. controls (-12,1) and (-12,-1) .. (b2);
\draw[edge] (a) .. controls (-7,1) and (-7,-1) .. (b);
\draw[edge,color=magenta,very thick] (a) -- (b);
\draw[edge,color=orange,very thick] (a2) -- (b2);
\draw[edge] (a) -- (a2);
\draw[edge] (b) -- (b2);

\path [thick,draw=blue,style={decorate, decoration=snake}] (-5.5,0) -- (-3.7,0);
\path [thick,draw=blue,->] (-3.7,0) -- (-3.5,0);
\foreach \i in {0,...,5} {
  \node (u_\i) at (-90+\k*\i:2.0) {};
}
\foreach \i in {0,...,3} {
  \node (d_\i) at ($sqrt(3)*(2,0) + (-90+\k*\i:2.0)$) {};
}
\draw[edge] (u_1) -- (d_0) -- (d_1); 
\draw[edge,color=magenta,very thick] (d_1) -- (d_2);
\draw[edge] (d_2) -- (d_3) -- (u_2);
\draw[edge] (u_0) -- (u_1);
\draw[edge,color=orange,very thick] (u_1) -- (u_2);
\draw[edge] (u_2) -- (u_3) -- (u_4) -- (u_5) -- (u_0);
\foreach \i in {0,...,3} {
  \node (e_\i) at ($sqrt(3)*(4,0) + (-90+\k*\i:2.0)$) {};
}
\draw[edge] (d_1) -- (e_0) -- (e_1) -- (e_2) -- (e_3) -- (d_2);
\end{tikzpicture}

%% file: figure_6c.tikz
\usetikzlibrary{calc}
\usetikzlibrary{decorations.pathmorphing}
\usetikzlibrary{arrows}
\begin{tikzpicture}[scale=0.5,>=angle 60]
\tikzstyle{every node} = [inner sep=2, draw, circle, fill=blue!50]
\tikzstyle{edge} = [draw, line width=1.0]
\pgfmathsetmacro{\k}{360 / 6};
\node (a) at (-8, 2) {};
\node (b) at (-8, -2) {};
\node (a2) at (-11, 2) {};
\node (b2) at (-11, -2) {};

\draw[edge] (a2) .. controls (-12,1) and (-12,-1) .. (b2);
\draw[edge] (a) .. controls (-7,1) and (-7,-1) .. (b);
\draw[edge,color=magenta,very thick] (a) -- (b);
\draw[edge,color=orange,very thick] (a2) -- (b2);
\draw[edge] (a) -- (a2);
\draw[edge] (b) -- (b2);

\path [thick,draw=blue,style={decorate, decoration=snake}] (-5.5,0) -- (-3.7,0);
\path [thick,draw=blue,->] (-3.7,0) -- (-3.5,0);
\foreach \i in {0,...,5} {
  \node (u_\i) at (-90+\k*\i:2.0) {};
}
\foreach \i in {0,...,3} {
  \node (d_\i) at ($sqrt(3)*(2,0) + (-90+\k*\i:2.0)$) {};
}
\draw[edge] (u_1) -- (d_0) -- (d_1) -- (d_2); 
\draw[edge,color=magenta,very thick] (d_2) -- (d_3);
\draw[edge] (d_3) -- (u_2);
\draw[edge] (u_0) -- (u_1);
\draw[edge,color=orange,very thick] (u_1) -- (u_2);
\draw[edge] (u_2) -- (u_3) -- (u_4) -- (u_5) -- (u_0);
\foreach \i in {0,...,3} {
  \node (e_\i) at ($sqrt(3)*(3,0) + (0, 3) + (-30+\k*\i:2.0)$) {};
}
\draw[edge] (d_2) -- (e_0) -- (e_1) -- (e_2) -- (e_3) -- (d_3);
\end{tikzpicture}

%% file: figure_6d.tikz
\usetikzlibrary{calc}
\usetikzlibrary{decorations.pathmorphing}
\usetikzlibrary{arrows}
\begin{tikzpicture}[scale=0.5,>=angle 60]
\tikzstyle{every node} = [inner sep=2, draw, circle, fill=blue!50]
\tikzstyle{edge} = [draw, line width=1.0]
\pgfmathsetmacro{\k}{360 / 6};

\node (o) at (-9, 1) {};
\node (a1) at ($ (-9, 1) + (0, -2) $) {};
\node (a2) at ($ (-9, 1) + (30:2) $) {};
\node (a3) at ($ (-9, 1) + (150:2) $) {};

\draw[edge,color=orange,very thick] (o) -- (a1);
\draw[edge,color=green,very thick] (o) -- (a2);
\draw[edge,color=magenta,very thick] (o) -- (a3);
\draw[edge] (a1) -- (a2) -- (a3) -- (a1);

\path [thick,draw=blue,style={decorate, decoration=snake}] (-5.5,0) -- (-3.7,0);
\path [thick,draw=blue,->] (-3.7,0) -- (-3.5,0);
\foreach \i in {0,...,5} {
  \node (u_\i) at (-90+\k*\i:2.0) {};
}
\foreach \i in {0,...,3} {
  \node (d_\i) at ($sqrt(3)*(2,0) + (-90+\k*\i:2.0)$) {};
}
\draw[edge] (u_1) -- (d_0) -- (d_1) -- (d_2) -- (d_3); 
\draw[edge,color=green,very thick] (d_3) -- (u_2);
\draw[edge] (u_0) -- (u_1);
\draw[edge,color=orange,very thick] (u_1) -- (u_2);
\draw[edge,color=magenta,very thick] (u_2) -- (u_3);
\draw[edge] (u_3) -- (u_4) -- (u_5) -- (u_0);
\foreach \i in {1,...,3} {
  \node (e_\i) at ($sqrt(3)*(1,0) + (0, 3) + (-30+\k*\i:2.0)$) {};
}
\draw[edge] (d_3) -- (e_1) -- (e_2) -- (e_3) -- (u_3);
\end{tikzpicture}

%% file: figure_7.tikz
\begin{tikzpicture}[scale=0.5,>=angle 60]
\tikzstyle{every node} = [inner sep=2, draw, circle, fill=blue!50]
\tikzstyle{edge} = [draw, line width=1.0]
\pgfmathsetmacro{\k}{360 / 6};

\foreach \i in {0,...,2} {
  \node (a_\i) at (120*\i:{4/sqrt(3)} ) {};
}
\draw[edge] (a_0) -- (a_1) -- (a_2) -- (a_0);
\draw[edge] (a_1) .. controls (-2.5,1) and (-2.5,-1) .. (a_2);

\foreach \i in {0,...,2} {
  \node (b_\i) at ($({8/sqrt(3)+4},0) + (60+120*\i:{4/sqrt(3)} )$) {};
}
\draw[edge] (b_0) .. controls ($({8/sqrt(3)+4},0) + (2.5,1)$) and ($({8/sqrt(3)+4},0) + (2.5,-1)$) .. (b_2);
\draw[edge] (b_0) -- (b_1) -- (b_2) -- (b_0);
\draw[edge] (a_0) -- (b_1);
\end{tikzpicture}

%% file: figure_9.tikz
\usetikzlibrary{calc}
\usetikzlibrary{decorations.pathmorphing}
\usetikzlibrary{arrows}
\begin{tikzpicture}[scale=0.5,>=angle 60]
\tikzstyle{every node} = [inner sep=2, draw, circle, fill=blue!50]
\tikzstyle{edge} = [draw, line width=1.0]
\pgfmathsetmacro{\k}{360 / 6};
  \node (u_0) at (-90+\k*0:2.0) {};
  \node (u_1) at (-90+\k*1:2.0) {};
  \node (u_2) at ($(-90+\k*2:2.0) + (0, -0.25)$) {};
  \node (u_3) at ($(-90+\k*3:2.0) + (0, -0.25)$) {};
  \node (u_4) at (-90+\k*4:2.0) {};
  \node (u_5) at (-90+\k*5:2.0) {};
\draw[edge] (u_0) -- (u_1) -- (u_2) -- (u_3) -- (u_4) -- (u_5) -- (u_0);
\foreach \i in {1,...,3} {
  \node (d_\i) at ($sqrt(3)*(2,0) + (-90+\k*\i:2.0)$) {};
}
  \node (d_0) at ($sqrt(3)*(2,0) + (-90+\k*0:2.0)$) {};
\draw[edge] (u_1) -- (d_0) -- (d_1) -- (d_2) -- (d_3);
\foreach \i in {2,...,3} {
  \node (e_\i) at ($sqrt(3)*(4,0) + (-90+\k*\i:2.0)$) {};
}
 \node (e_0) at ($sqrt(3)*(4,0) + (-90+\k*0:2.0)$) {};
 \node (e_1) at ($sqrt(3)*(4,0) + (-90+\k*1:2.0)$) {};
\draw[edge] (d_1) -- (e_0) -- (e_1) -- (e_2) -- (e_3) -- (d_2);
\foreach \i in {0,...,2} {
  \node (a_\i) at ($sqrt(3)*(3,0) + (0, 3) + (30+\k*\i:2.0)$) {};
}
\draw[edge] (e_3) -- (a_0) -- (a_1) -- (a_2) -- (d_3);
\foreach \i in {0,...,1} {
  \node (b_\i) at ($sqrt(3)*(1,0) + (0, 3) + (90+\k*\i:2.0)$) {};
}
\node (g_0) at ($(u_0) + (-90:2)$) {}; \draw[edge] (g_0) -- (u_0);
\node (g_1) at ($(d_0) + (-90:2)$) {}; \draw[edge] (g_1) -- (d_0);
\node (g_2) at ($(e_0) + (-90:2)$) {}; \draw[edge] (g_2) -- (e_0);
\node (h_0) at ($(g_0) + (-30:2)$) {}; \draw[edge] (g_0) -- (h_0) -- (g_1);
\node (h_1) at ($(g_1) + (-30:2)$) {}; \draw[edge] (g_1) -- (h_1) -- (g_2);
 \node (uuu_2) at ($(u_2) + (0, 0.5)$) {};
\node (uuu_3) at ($(u_3) + (0, 0.5)$) {};
\draw[edge] (a_2) -- (b_0) -- (b_1) -- (uuu_3) -- (uuu_2);
\draw[edge] (uuu_2) -- (d_3);
\end{tikzpicture}

%% file: figure_12.tikz
\usetikzlibrary{calc}
\usetikzlibrary{intersections}
\begin{tikzpicture}[scale=0.5]
\tikzstyle{every node} = [inner sep=2, draw, circle, fill=blue!50]
\tikzstyle{edge} = [draw, line width=1.0]
\pgfmathsetmacro{\k}{360 / 6};
\draw[thick, dashed, black] ($(-90+\k*1:2.0) + (-30:2) + (30:2) + (-30:2) + (30:2) + (90:2)$) arc (-30:210:6cm and 1.3cm);
\draw[name path=elipsa2,thick, dashed, black] ($(-90+\k*1:2.0) + (-30:2) + (30:2) + (-30:2) + (30:2)$) arc (-30:210:6cm and 1.3cm);
\fill[white] ($sqrt(3)*(-1,0) + (0,-2)$) rectangle ($sqrt(3)*(5,0) + (0, 2)$);
\foreach \i in {0,...,5} {
  \node (u_\i) at (-90+\k*\i:2.0) {};
}
\foreach \i in {0,...,5} {
  \pgfmathtruncatemacro{\j}{mod(\i+1, 6)};
  \draw[edge] (u_\i) -- (u_\j);
}
\node (v_0) at ($(u_1) + (-30:2)$) {};
\node (v_1) at ($(v_0) + (30:2)$) {};
\node (v_2) at ($(v_1) + (90:2)$) {};
\node (v_3) at ($(v_2) + (150:2)$) {};
\node (w_0) at ($(v_1) + (-30:2)$) {};
\node (w_1) at ($(w_0) + (30:2)$) {};
\node (w_2) at ($(w_1) + (90:2)$) {};
\node (w_3) at ($(w_2) + (150:2)$) {};
\draw[edge] (u_1) -- (v_0) -- (v_1) -- (v_2) -- (v_3) -- (u_2);
\draw[edge] (v_1) -- (w_0) -- (w_1) -- (w_2) -- (w_3) -- (v_2);
\end{tikzpicture}

%% file: figure_11.tikz
\usetikzlibrary{calc}
\usetikzlibrary{decorations.pathmorphing}
\usetikzlibrary{decorations.pathreplacing}
\begin{tikzpicture}[scale=0.4]
\tikzstyle{every node} = [inner sep=2, draw, circle, fill=blue!50]
\tikzstyle{edge} = [draw, line width=1.0]
\tikzstyle{spoke} = [draw, line width=1.0, decorate, decoration={snake, segment length=2mm, amplitude=1.2pt}, color=blue]
\pgfmathsetmacro{\k}{360 / 6};
\foreach \i in {0,...,5} {
  \node (u_\i) at (30+\k*\i:2.0) {};
}
\foreach \i in {0,...,5} {
  \pgfmathtruncatemacro{\j}{mod(\i+1, 6)};
  \draw[edge] (u_\i) -- (u_\j);
}
\node (a_0) at ($(u_0) + (30:2)$) {};
\node (a_1) at ($(a_0) + (-30:2)$) {};
\node (a_2) at ($(a_1) + (-90:2)$) {};
\node (a_3) at ($(a_2) + (210:2)$) {};
\draw[edge] (u_0) -- (a_0) -- (a_1) -- (a_2) -- (a_3) -- (u_5);
\node (b_0) at ($(a_1) + (30:2)$) {};
\node (b_1) at ($(b_0) + (-30:2)$) {};
\node (b_2) at ($(b_1) + (-90:2)$) {};
\node (b_3) at ($(b_2) + (210:2)$) {};
\draw[edge] (a_1) -- (b_0) -- (b_1) -- (b_2) -- (b_3) -- (a_2);
\node (g_0) at ($(u_1) + (90:2)$) {}; \draw[edge] (g_0) -- (u_1);
\node (g_1) at ($(a_0) + (90:2)$) {}; \draw[edge] (g_1) -- (a_0);
\node (g_2) at ($(b_0) + (90:2)$) {}; \draw[edge] (g_2) -- (b_0);
\node (gg_2) at ($(b_1) + (30:2)$) {}; \draw[edge] (gg_2) -- (b_1);
\node (g_3) at ($(gg_2) + (90:2)$) {}; \draw[edge] (gg_2) -- (g_3);
\node (c_0) at ($(g_0) + (30:2)$) {}; \draw[edge] (g_0) -- (c_0) -- (g_1);
\node (c_1) at ($(g_1) + (30:2)$) {}; \draw[edge] (g_1) -- (c_1) -- (g_2);
\node (c_2) at ($(g_2) + (30:2)$) {}; \draw[edge] (g_2) -- (c_2) -- (g_3);
\node (cc_0) at ($(g_0) + (150:2)$) {}; \draw[edge] (g_0) -- (cc_0);
\node (d_0) at ($(c_0) + (90:2)$) {}; \draw[edge] (c_0) -- (d_0);
\node (d_1) at ($(c_1) + (90:2)$) {}; \draw[edge] (c_1) -- (d_1);
\node (d_2) at ($(c_2) + (90:2)$) {}; \draw[edge] (c_2) -- (d_2);
\node (dd_0) at ($(cc_0) + (90:2)$) {}; \draw[edge] (cc_0) -- (dd_0);
\node (n_0) at ($(dd_0) + (30:2)$) {}; \draw[edge] (dd_0) -- (n_0) -- (d_0);
\node (n_1) at ($(d_0) + (30:2)$) {}; \draw[edge] (d_0) -- (n_1) -- (d_1);
\node (n_2) at ($(d_1) + (30:2)$) {}; \draw[edge] (d_1) -- (n_2) -- (d_2);

\foreach \i in {0,...,5} {
  \node (w_\i) at ($(30+\k*\i:2.0) + (19, 0)$) {};
}
\foreach \i in {0,...,5} {
  \pgfmathtruncatemacro{\j}{mod(\i+1, 6)};
  \draw[edge] (w_\i) -- (w_\j);
}
\node (aa_0) at ($(w_0) + (30:2)$) {};
\draw[edge] (aa_0) -- (w_0);
\node (h_0) at ($(aa_0) + (90:2)$) {};
\node (h_1) at ($(w_1) + (90:2)$) {};
\draw[edge] (aa_0) -- (h_0);
\draw[edge] (w_1) -- (h_1);
\node (h2_0) at ($(h_0) + (150:2)$) {};
\node (h2_1) at ($(h_1) + (150:2)$) {};
\draw[edge] (h_0) -- (h2_0) -- (h_1);
\draw[edge] (h_1) -- (h2_1); 
\node (h3_0) at ($(h2_0) + (90:2)$) {};
\node (h3_1) at ($(h2_1) + (90:2)$) {};
\draw[edge] (h3_1) -- (h2_1);
\draw[edge] (h3_0) -- (h2_0);
\node (h4) at ($(h3_1) + (30:2)$) {};
\draw[edge] (h3_0) -- (h4) -- (h3_1);

\node[fill=none,draw=none] at (13, 0) {$\cdots$};
\node[fill=none,draw=none] at (13, 6) {$\cdots$};

\draw [line width=0.8,decorate,decoration={brace,amplitude=12pt}]
($(h4) + (4.5, 0)$) -- ($(w_4) + (4.5, 0)$) node [fill=none,draw=none,midway,xshift=20pt]  {$s$};

\draw [line width=0.8,decorate,decoration={brace,amplitude=12pt}]
($(dd_0) + (0, 2)$) -- ($(h3_0) + (0, 2)$) node [fill=none,draw=none,midway,yshift=20pt]  {$k$};

\end{tikzpicture}

%% file: kapica_5_6.tikz
\begin{tikzpicture}[scale=0.3]
\tikzstyle{every node} = [inner sep=2, draw, circle, fill=blue!50]
\tikzstyle{edge} = [draw, line width=1.0]
\pgfmathsetmacro{\k}{360 / 5};
\foreach \i in {0,...,4} {
  \node (u_\i) at (90+\k*\i:2.0) {};
  \node (v_\i) at (90+\k*\i:4.0) {};
  \draw[edge] (u_\i) -- (v_\i);
}
\foreach \i in {0,...,4} {
  \pgfmathtruncatemacro{\j}{mod(\i+1, 5)};
  \draw[edge] (u_\i) -- (u_\j);
}
\pgfmathsetmacro{\kt}{\k / 3};
\foreach \i in {0,...,4} {
  \pgfmathtruncatemacro{\j}{mod(\i+1, 5)};
  \node (w_\i) at (90+\k*\i+\kt:5.0) {};
  \node (z_\i) at (90+\k*\i+2*\kt:5.0) {};
  \draw[edge] (v_\i) -- (w_\i) -- (z_\i) -- (v_\j);
}
\end{tikzpicture}

%% file: kapica_5_6_altansemi.tikz
\usetikzlibrary{calc}
\begin{tikzpicture}[scale=0.3]
\tikzstyle{every node} = [inner sep=2, draw, circle, fill=blue!50]
\tikzstyle{edge} = [draw, line width=1.0]
\pgfmathsetmacro{\k}{360 / 5};
\foreach \i in {0,...,4} {
  \node (u_\i) at (90+\k*\i:2.0) {};
  \node (v_\i) at (90+\k*\i:4.0) {};
  \draw[edge] (u_\i) -- (v_\i);
}
\foreach \i in {0,...,4} {
  \pgfmathtruncatemacro{\j}{mod(\i+1, 5)};
  \draw[edge] (u_\i) -- (u_\j);
}
\pgfmathsetmacro{\kt}{\k / 3};
\foreach \i in {0,...,4} {
  \pgfmathtruncatemacro{\j}{mod(\i+1, 5)};
  \node (w_\i) at (90+\k*\i+\kt:5.0) {};
  \node (z_\i) at (90+\k*\i+2*\kt:5.0) {};
  \draw[edge] (v_\i) -- (w_\i) -- (z_\i) -- (v_\j);
}
\foreach \i in {0,...,4} {
	\node[fill=none,draw=none] (a_\i) at ($(w_\i) + (-10+90+\k*\i+\kt:2.0)$) {}; \draw[edge] (a_\i) -- (w_\i);
    \node[fill=none,draw=none] (b_\i) at ($(z_\i) + (10+90+\k*\i+2*\kt:2.0)$) {}; \draw[edge] (b_\i) -- (z_\i);
}
\foreach \i in {0,...,4} {
  \pgfmathtruncatemacro{\j}{mod(\i+1, 5)};

}
\end{tikzpicture}

%% file: kapica_5_6_altan.tikz
\usetikzlibrary{calc}
\begin{tikzpicture}[scale=0.3]
\tikzstyle{every node} = [inner sep=2, draw, circle, fill=blue!50]
\tikzstyle{edge} = [draw, line width=1.0]
\pgfmathsetmacro{\k}{360 / 5};
\foreach \i in {0,...,4} {
  \node (u_\i) at (90+\k*\i:2.0) {};
  \node (v_\i) at (90+\k*\i:4.0) {};
  \draw[edge] (u_\i) -- (v_\i);
}
\foreach \i in {0,...,4} {
  \pgfmathtruncatemacro{\j}{mod(\i+1, 5)};
  \draw[edge] (u_\i) -- (u_\j);
}
\pgfmathsetmacro{\kt}{\k / 3};
\foreach \i in {0,...,4} {
  \pgfmathtruncatemacro{\j}{mod(\i+1, 5)};
  \node (w_\i) at (90+\k*\i+\kt:5.0) {};
  \node (z_\i) at (90+\k*\i+2*\kt:5.0) {};
  \draw[edge] (v_\i) -- (w_\i) -- (z_\i) -- (v_\j);
}
\foreach \i in {0,...,4} {
	\node (a_\i) at ($(w_\i) + (-10+90+\k*\i+\kt:2.0)$) {}; \draw[edge] (a_\i) -- (w_\i);
    \node (b_\i) at ($(z_\i) + (10+90+\k*\i+2*\kt:2.0)$) {}; \draw[edge] (b_\i) -- (z_\i);
    \node (uu_\i) at ($(u_\i) + (90+\k*\i:6.0)$) {};
   \node (vv_\i) at (-36+90+\k*\i:8.0) {};
}
\foreach \i in {0,...,4} {
  \pgfmathtruncatemacro{\j}{mod(\i+1, 5)};
  \draw[edge] (a_\i) -- (vv_\j);
  \draw[edge] (b_\i) -- (vv_\j);
  \draw[edge] (a_\j) -- (uu_\j);
  \draw[edge] (b_\i) -- (uu_\j);

}
\end{tikzpicture}

%% file: kapica_5_6_altan2semi.tikz
\usetikzlibrary{calc}
\begin{tikzpicture}[scale=0.3]
\tikzstyle{every node} = [inner sep=2, draw, circle, fill=blue!50]
\tikzstyle{edge} = [draw, line width=1.0]
\pgfmathsetmacro{\k}{360 / 5};
\foreach \i in {0,...,4} {
  \node (u_\i) at (90+\k*\i:2.0) {};
  \node (v_\i) at (90+\k*\i:4.0) {};
  \draw[edge] (u_\i) -- (v_\i);
}
\foreach \i in {0,...,4} {
  \pgfmathtruncatemacro{\j}{mod(\i+1, 5)};
  \draw[edge] (u_\i) -- (u_\j);
}
\pgfmathsetmacro{\kt}{\k / 3};
\foreach \i in {0,...,4} {
  \pgfmathtruncatemacro{\j}{mod(\i+1, 5)};
  \node (w_\i) at (90+\k*\i+\kt:5.0) {};
  \node (z_\i) at (90+\k*\i+2*\kt:5.0) {};
  \draw[edge] (v_\i) -- (w_\i) -- (z_\i) -- (v_\j);
}
\foreach \i in {0,...,4} {
	\node (a_\i) at ($(w_\i) + (-10+90+\k*\i+\kt:2.0)$) {}; \draw[edge] (a_\i) -- (w_\i);
    \node (b_\i) at ($(z_\i) + (10+90+\k*\i+2*\kt:2.0)$) {}; \draw[edge] (b_\i) -- (z_\i);
    \node (uu_\i) at ($(u_\i) + (90+\k*\i:6.0)$) {};
   \node (vv_\i) at (-36+90+\k*\i:8.0) {};
}
\foreach \i in {0,...,4} {
  \pgfmathtruncatemacro{\j}{mod(\i+1, 5)};
  \draw[edge] (a_\i) -- (vv_\j);
  \draw[edge] (b_\i) -- (vv_\j);
  \draw[edge] (a_\j) -- (uu_\j);
  \draw[edge] (b_\i) -- (uu_\j);
}
\foreach \i in {0,...,4} {
   \node[draw=none,fill=none] (yy_\i) at (-36+90+\k*\i:10.0) {}; \draw[edge] (yy_\i) -- (vv_\i);
    \node[draw=none,fill=none] (xx_\i) at ($(uu_\i) + (90+\k*\i:2.0)$) {}; \draw[edge] (xx_\i) -- (uu_\i);

}
\end{tikzpicture}